\documentclass[letterpaper, 10 pt, conference]{ieeeconf}  

\IEEEoverridecommandlockouts                              
\usepackage{graphicx} 
\usepackage{epsfig} 
\usepackage{amsmath} 
\usepackage{amsthm}
\usepackage{amssymb}  
\let\labelindent\relax
\usepackage{enumitem}
\usepackage[hidelinks]{hyperref}
\usepackage{cleveref}
\usepackage{xcolor}
\usepackage{comment}
\usepackage{subcaption}
\newtheorem{thm}{Theorem}

\theoremstyle{definition}
\newtheorem{defn}{Definition}

\newtheorem{rem}{Remark}

\allowdisplaybreaks[1]

\title{\LARGE \bf Solving Monotone Variational Inequalities with Best Response Dynamics
}

\author{Yu-Wen Chen, Can Kizilkale, Murat Arcak
\thanks{Y.-W. Chen,  C. Kizilkale and M. Arcak are with the Department of Electrical Engineering and Computer Sciences, University of California, Berkeley,
CA, USA.
{\tt\small \{yuwen\_chen, cankizilkale, arcak\}@berkeley.edu}.
This work was supported  by the National Science Foundation grant CNS-2135791.}%
}

\usepackage{tikz}

\newcommand\copyrighttext{%
  \footnotesize{\color{blue} \textcopyright \the\year{} IEEE. Personal use of this material is permitted. Permission from IEEE must be obtained for all other uses, including reprinting/republishing this material for advertising or promotional purposes, collecting new collected works for resale or redistribution to servers or lists, or reuse of any copyrighted component of this work in other works.}}

\newcommand\copyrightnotice{%
\begin{tikzpicture}[remember picture,overlay]
\node[anchor=north,yshift=-10pt] at (current page.north) {\fbox{\parbox{\dimexpr0.9\textwidth-\fboxsep-\fboxrule\relax}{\copyrighttext}}};
\end{tikzpicture}%
}
\renewcommand\fbox{\fcolorbox{red}{white}}
\begin{document}

\maketitle

\copyrightnotice
\vspace{-10pt}
\thispagestyle{empty}
\pagestyle{empty}

\begin{abstract}

We leverage best response dynamics to solve monotone variational inequalities on compact and convex sets.
Specialization of the method to variational inequalities in game theory recovers 
convergence results to Nash equilibria 
when agents select the best response to the current distribution of strategies.
We apply the method to generalize population games with additional convex constraints. Furthermore, we explore the robustness of the method by introducing various types of time-varying disturbances. 

\end{abstract}
\section{INTRODUCTION}\label{sec:Intro}

Variational inequalities, introduced in \cite{Lions1967}, cover wide-ranging problems in modeling equilibria and optimization problems featuring inequality constraints. In transportation, they represent traffic equilibria \cite{Dafermos1980}. In operations research, they model equilibria in supply chains, network flows, and facility locations \cite{Allevi2018}. In finance, they reflect equilibrium prices, offering insights into financial market dynamics \cite{Nagurney2013}. In game theory, they capture Nash equilibria of noncooperative games \cite{Facchinei2003}. For atomic games, the Nash equilibria can be captured by variational inequalities derived from first order optimality conditions. In nonatomic (population) games, the equilibrium condition comes directly in the form of a variational inequality \cite{Sorin2016}. 
Many generalizations of variational inequalities have been proposed, including the complementary problem \cite{Lemke1965}, the quasi variational inequality \cite{Chan1982} and the general variational inequality \cite{Noor1988}.

To solve variational inequalities, constrained gradient descent methods are widely adopted in the literature \cite{Hieu2020,Gidel2017,Allibhoy2022}. 
These methods iteratively update the solution by incorporating both the gradient information of the objective function and the constraints to ensure that the iterates remain feasible. One such method is projected gradient descent, which optimizes constrained problems by iteratively adjusting solutions along the negative gradient direction and projecting them onto the feasible set. Another variant, the Frank-Wolfe method \cite{Frank1956}, solves the problem by iteratively moving towards the solution along linear approximations of the objective function. 
The Augmented Lagrangian Method \cite{Nocedal2006} solves constrained optimization problems by iteratively updating the Lagrange multipliers while minimizing a penalized version of the objective function. It combines the benefits of penalty methods and the method of multipliers. 


Finding the Nash equilibria in noncooperative games, as an application of variational inequalities, has its own track of research. 
In particular, the best response dynamics is one of the incentive-oriented methods, describing the iterative process where players sequentially update their strategies to 
minimize their cost based on the current strategies of others. If the process stops, meaning that no player benefits from changing unilaterally, then the Nash equilibrium is reached. The existing literature 
examines the 
convergence characteristics of the best response dynamics in game theory, 
notably in potential games as outlined in \cite{Monderer1996}. 
The best response dynamics for population games follow the method proposed in \cite{Matsui1992}, where the best response describes the best strategy distribution to the present 
cost.
The convergence results of the best response dynamics for population games have been well-studied for particular 
cost structures \cite{Hofbauer2009, Sandholm2010, Fox2013, Arcak2020}.

The first contribution of this paper is to leverage the best response dynamics to solve general variational inequalities, recovering results for 
games as special cases.
Furthermore, the method extends the scope of population games by allowing additional convex constraints, thus extending previous work for standard population games \cite{Hofbauer2009}.

The second contribution is to account for disturbances in  best response dynamics, moving beyond static forms of disturbance typically studied in variational inequalities. We introduce various types of time-varying disturbances and 
prove appropriate forms of 
Input-to-State Stability \cite{Sontag1989}. 

The remainder of the paper is organized as follows. In \Cref{sec:Pre}, preliminary results are given. In \Cref{sec:BRD}, solutions of variational inequalities are studied using the best response dynamics. In \Cref{sec:Dis}, various types of disturbance are introduced and the robustness of the best response dynamics is analyzed. In \Cref{sec:APP}, several applications are presented. Finally in \Cref{sec:Conclusion}, conclusions are given.

\section{PRELIMINARIES}\label{sec:Pre}
In this section, we review variational inequalities and the best response dynamics.

\subsection{Variational Inequalities} \label{sec:pre_vi}
\begin{defn}[Variational Inequality] \label{def:VI}
    Given a set $K\subseteq\mathbb{R}^n$ and a mapping $F:K\rightarrow\mathbb{R}^n$, we say that $x\in K$ solves the variational inequality, denoted as $VI(K,F)$, 
    if
    \begin{align}
    (y-x)^T F(x) \geq 0 , \quad \forall y \in K. \label{eq:VI}
    \end{align}
    The set of all solutions is denoted as $SOL(K,F)$.
\end{defn}

\begin{defn}[Strong Monotonicity / Monotonicity] \label{def:strong_mono}
    A function $F:K\subseteq\mathbb{R}^n\rightarrow\mathbb{R}^n$ is strongly monotone (or monotone) on $K$ if $\exists \ c>0$ (or $c=0$) s.t.
    \begin{align}
        (x-y)^T\left(F(x)-F(y)\right)\geq c||x-y||^2, \quad \forall x,y\in K.
    \end{align}\label{eq:strong_mono}
    If $F$ is continuously differentiable, then (\ref{eq:strong_mono}) is equivalent to \begin{align}
    z^T DF(x) z \geq c||z||^2, \quad \forall x\in K, z\in TK, \label{eq:PD}
\end{align}
where $TK$ denotes the tangent space to $K$.
\end{defn}


Most existing solvers for variational inequalities are closely connected to constrained optimization. Especially, lots of them are based on the projected gradient iterative algorithms; {\it e.g.},
\cite{Nagurney2012} proposes the continuous version,
\begin{align}
    \dot{x} &= \Pi_{TK(x)}(-F(x)), \label{eq:VI_alg_conti}
\end{align}
where $TK(x)$ denotes the tangent cone at $x$.

\subsection{Best response dynamics}\label{sec:pre_brd}


\subsubsection{Best response dynamics in atomic games}\label{sec:pre_brd_atomic}

Consider $n$ agents. For $i=1$ to $n$, agent-$i$ has a $d_i$-dimensional optimization variable  $x_i\in\mathbb{R}^{d_i}$ and an objective function $f_i(x_i,x_{-i})$ with a constraint set $\mathcal{X}_i\subseteq\mathbb{R}^{d_i}$, where $x_{-i}$ denotes all the optimization variables except $x_i$. Then, the best response is defined as
\begin{align}
    \arg\min_{x_i\in \mathcal{X}_i}\ f_i(x_i,x_{-i}). \label{eq:BR_atomic}
\end{align}
That is, the best response of an agent refers to a strategy that optimizes its objective function given the actions taken by the other agents.


The best response dynamics 
describes an iterative process where agents take turns making their best responses.
This process iterates until no player has an incentive to deviate from its chosen strategy unilaterally. Thus, the resulting strategies form a Nash equilibrium. The Nash equilibrium is a solution to the variational inequality (\ref{eq:VI}) where $F=[\nabla f_1,\cdots,\nabla f_n]^T$ and $K$ is the Cartesian product set $\mathcal{X}_1\times\cdots\times\mathcal{X}_n$ \cite{Pavel2022}.

\subsubsection{Best response dynamics in population games}\label{sec:pre_brd_pop}

In a single population game, the set of strategies available to the agents is denoted as $\mathcal{S}=\{1,\cdots,n\}$. Then, the social state is defined as $x=(x_1, \cdots, x_n)\in \mathbb{R}^n_+$, where $x_i$ represents the fraction of players choosing strategy $i$. Note that the social state $x$ 
lies
in the probability simplex $X=\{v=(v_1,\cdots,v_n)\in\mathbb{R}^n_+: \sum_{i=1}^n v_i=1, v_i\geq0\}$. A cost function $G:X \rightarrow \mathbb{R}^n$ maps the social state $x$ to a cost, denoted as $\pi=G(x)$. 
%
For a state $x\in X$, the best response to cost $\pi=G(x)$ is the strategy distribution that minimizes the total cost,
\begin{align*}
    \arg\min_{y\in X}\ y^T\pi.
\end{align*}

Evolutionary dynamics models for population games assume each agent continually revises its strategy and revision opportunities follow a Poisson process. When the opportunity arises, the agent switches its strategy according to a probability distribution defined by a learning rule \cite{Sandholm2010}. When the rule is to select the strategy with the lowest cost (``best response"), the resulting mean dynamics are
\begin{align*}
    \dot{x}(t) \in \left(\arg\min_{y\in X}\ y^T\pi(t) \right)-x(t),
\end{align*}
where $\pi(t)=G(x(t))$.


The solution of the best response dynamics above has been proven to converge to Nash equilibria in various games, {\it e.g.}, potential games, monotone games, and supermodular games; see \cite{Sandholm2010} for a comprehensive review. 

\section{BEST RESPONSE DYNAMICS FOR VARIATIONAL INEQUALITIES}\label{sec:BRD}

We now generalize
the best response dynamics 
to solve broader
variational inequalities than those describing Nash equilibria in population games. 
To do so, we extend the feasible set of the best response from a probability simplex $X$ to an arbitrary compact and convex set $K$.


\begin{defn}[Best response] \label{def:BR_BRD}
Let $K\subseteq\mathbb{R}^n$ be a compact and convex set. Given a vector $\pi\in\mathbb{R}^n$, the best response mapping $\beta(\pi)$ is defined as
\begin{align}
    \beta(\pi) = \arg\min_{y\in K}\ y^T \pi. \label{eq:BR} 
\end{align}
\end{defn}
Note that $\beta$ is a set-valued map. 
Next we define the best response dynamics.
\begin{defn}[Best response dynamics]
    Given a cost trajectory $\pi(t)$, the best response dynamics 
    is the
    differential inclusion,
    \begin{align}
        \dot{x}(t) \in \beta(\pi(t)) - x(t), \quad \forall t\geq0. \label{eq:BRD}
    \end{align}
\end{defn}
 Since $K$ is convex, $\beta(\pi(t))-x(t)$ belongs to the tangent cone at $x(t)$ making $K$ invariant under (\ref{eq:BRD}). Therefore, we do not require projection steps onto $K$ to ensure feasibility as in projected gradient methods. 

In the following, 
we use the best response dynamics to solve the variational inequality $VI(K,F)$.

\begin{thm} \label{thm:convex}
    Consider the variational inequality (\ref{eq:VI}) where $K \subseteq\mathbb{R}^n$ is compact and convex and $F:K \rightarrow \mathbb{R}^n$ is $C^1$ monotone. Then SOL($K$,$F$) is globally asymptotically stable under the best response dynamics (\ref{eq:BRD}) with $\pi(t)=F(x(t))$.
\end{thm}

\begin{proof}
Define the mapping on the right hand side of (\ref{eq:BRD}) as $f(x):=\beta(F(x))-x$. Since $f$ is upper-hemicontinous, nonempty, compact-valued, and convex-valued, there exists a Carath{\'e}odory solution $x(t), \forall t\geq0$ \cite{Aubin1984}.
   
   Define the Lyapunov function candidate
\begin{align}
    V(x) = U(x,F(x)),\label{eq:trueV}
\end{align}
where
\begin{align}
    U(x,\pi) = x^T \pi - \underbrace{ \min_{y\in K} \ y^T \pi}_{:= m(\pi)}, \label{eq:V}
\end{align}
which is nonnegative on $K$ and vanishes only on $SOL(K,F)$. To apply the nonsmooth Lyapunov techniques \cite{Shevitz1994}, we need to show that (\ref{eq:V}) is Lipschitz continuous w.r.t. its first and second arguments, respectively. The Lipschitzness w.r.t the first argument is from the affine form of (\ref{eq:V}). 
To prove the Lipschitzness of $m(\pi)$ w.r.t. $\pi$, we note that
    \begin{align*}
                m(\pi_1)-m(\pi_2)&=\min_{y\in K} \ y^T \pi_1 - \min_{y\in K} \ y^T \pi_2\\
                &\geq \min_{y\in K}\ y^T(\pi_1-\pi_2)
                \geq -M ||\pi_1-\pi_2||,
    \end{align*}
    where $M=\max_{x\in K}||x||$. Similarly, $m(\pi_2)-m(\pi_1)\geq -M||\pi_2-\pi_1||$, then the Lipschitz property: $|m(\pi_1)-m(\pi_2)|\leq M ||\pi_1-\pi_2||$  follows. Therefore, (\ref{eq:V}) is Lipschitz w.r.t. its second argument. Moreover, $x(t)$ is absolutely continuous w.r.t. $t$ and, thus, so is $\pi(t)$. As a result, for almost all points where $V$ is differentiable and $\dot{x}$ exists, we have
\begin{align}
    \frac{d}{dt}V(x(t))&=\frac{\partial U}{\partial x}(x(t),\pi(t)) \dot{x}(t) + \frac{\partial U}{\partial \pi}(x(t),\pi(t)) \dot{\pi}(t)\nonumber\\
    &=\pi(t)^T \dot{x}(t) + x(t)^T\dot{\pi}(t)-\frac{\partial m}{\partial \pi}(\pi(t))\dot{\pi}(t)\nonumber\\
    &\stackrel{(a)}{=} \pi(t)^T \dot{x}(t) + \left(x(t)-\beta(\pi(t))\right)^T\dot{\pi}(t)\nonumber\\
    &=\pi(t)^T \dot{x}(t) - \underbrace{\dot{x}(t) \ DF(x(t)) \ \dot{x}(t)}_{\geq0, \text{ by (\ref{eq:PD})}}\nonumber\\
    &\leq \pi(t)^T \left(- x(t) + \beta(\pi(t))\right)
    = -V(x(t)), \label{eq:Vdot}
\end{align} 
where $(a)$ follows from the Envelope Theorem \cite{Milgrom2002}. Using Clarke's generalized gradient \cite{Clarke1975}, we can extend the inequality (\ref{eq:Vdot}) to hold for almost all $t$. Finally, by the Theorem A.2 in \cite{Hofbauer2009}, we have $V(x(t))\rightarrow0$. Therefore, $SOL(K,F)$ is globally asymptotically stable.
\end{proof}

When we specialize variational inequalities to population games, where $K$ is the probability simplex, $SOL(K,F)$ characterizes the Nash equilibria. Thus, the result in \cite{Hofbauer2009} is a special case of \Cref{thm:convex}.


A
connection between (\ref{eq:VI_alg_conti}) and (\ref{eq:BRD}) can be made as follows. 
If $F$ is the gradient of a function, then (\ref{eq:VI_alg_conti}) is the projected gradient descent method and (\ref{eq:BRD}) is the Frank-Wolfe method. 
Although
(\ref{eq:VI_alg_conti}) and (\ref{eq:BRD}) 
have comparable complexity, (\ref{eq:VI_alg_conti}) requires a projection onto the {\it varying} tangent cone at the current point $TK(x)$ to remain feasible but (\ref{eq:BRD}) only needs to select a point in the {\it static} set $K$ and automatically remains feasible. This feature brings benefits to analysis, especially when disturbances are considered in the following section.

\section{ROBUSTNESS ANALYSIS}\label{sec:Dis}

In \Cref{thm:convex}, the function $F$ is perfectly known. We now consider a time-varying disturbance $\Delta(t)$ which perturbs $F(x(t))$ into $F(x(t))+\Delta(t)$. Then, the best response dynamics become 
\begin{align}
    \dot{x}(t) \in \beta(\tilde{\pi}(t)) - x(t), \quad \tilde{\pi}(t)=F(x(t))+\Delta(t). \label{eq:pay_dis}
\end{align}

This form of the disturbance alters the trajectory $x(t)$ indirectly via changes to the cost function $F$. In contrast, a disturbance $\varepsilon(t)$ may affect $x(t)$ by directly perturbing the best response dynamics into
\begin{align}
    \dot{x}(t) \in \beta(\pi(t)) - x(t) + \varepsilon(t), \quad \pi(t)=F(x(t)). \label{eq:dyn_dis}
\end{align}

\begin{defn}[Cost disturbance and Dynamics disturbance]
    We call disturbances appearing as in (\ref{eq:pay_dis}) {\it cost disturbances} and those as in (\ref{eq:dyn_dis}) {\it dynamics disturbances}.  
\end{defn}


To analyze the dynamics disturbance in (\ref{eq:dyn_dis}), we assume that the disturbance does not violate the constraints. For example, in a congestion game, the distribution of the traffic flows over routes 
add up to the total demand despite disturbances in the dynamics governing the evolution of flows.

\begin{defn}[Admissible Dynamics Disturbance]\label{def:admissible} For (\ref{eq:dyn_dis}), a dynamics disturbance $\varepsilon(t)$ is {\it admissible} if it is piecewise continuous, bounded, and satisfies 
\begin{align*}
    \dot{x}(t) \in TK(x(t)), \quad \forall t\geq0,
\end{align*}
that is, the resulting trajectory $x(t)$ will remain in $K$.
\end{defn}

In the following, we explore the general case where the best response dynamics are subject to both dynamics disturbances and cost disturbances. We show that the effects of disturbances are captured by the notion of Input-to-State Stability. Denote comparison functions class-$\mathcal{K}$, class-$\mathcal{K}_\infty$, and class-$\mathcal{KL}$, by $\mathcal{K}$, $\mathcal{K}_\infty$ and $\mathcal{KL}$, respectively \cite{Khalil2002}.

\begin{thm}\label{thm:comb}
    Consider best response dynamics subject to a dynamics disturbance $\varepsilon(t)$ and a cost disturbance $\Delta(t)$, 
    \begin{align*}
        \dot{x}(t) \in \beta(\tilde{\pi}(t)) - x(t) + \varepsilon(t),\quad \tilde{\pi}(t)=F(x(t))+\Delta(t),
    \end{align*}
     where $\varepsilon(t)$ is admissible and $\Delta(t)$ is bounded and with a bounded derivative.
     Let $K \subseteq\mathbb{R}^n$ be compact and convex and $F:K \rightarrow \mathbb{R}^n$ be $C^1$ strongly monotone. Then,
     \begin{align}
        ||x(t)&-x^*||\leq \omega(||x(0)-x^*||, t) \nonumber\\ &+\gamma_1\left(\max\left\{||\varepsilon||_\infty,||\Delta||_\infty\right\}\right) + \gamma_2(||\dot{\Delta}||_\infty),\label{eq:comb_bound}
     \end{align}
     where $x^*$ is the unique solution of $VI(K,F)$, $\omega\in\mathcal{KL}$, and $\gamma_1,\gamma_2\in\mathcal{K}$. In particular, if $\varepsilon(t)\rightarrow0$, $\Delta(t)\rightarrow0$, and $\dot{\Delta}(t)\rightarrow0$, then $x(t)\rightarrow x^*$.
\end{thm}

\begin{proof}
    Similar to the proof of \Cref{thm:convex}, we define
    \begin{align}
        V_1(x)&=U_1(x,F(x)),\\
        V_2(x,t)&=U_2(x,F(x),\Delta(t)),
    \end{align}
    where
    \begin{align}
        U_1(x,\pi)&=x^T \pi - \min_{y\in K}\ y^T\pi,\label{eq:V1}\\
        U_2(x,\pi,\Delta)&=x^T (\pi+\Delta) - \min_{y\in K}\ y^T(\pi+\Delta).\label{eq:V2}
    \end{align}    
    We first derive the following relations:
    \begin{align}
        V_1(x) &\geq \alpha_1(||x-x^*||),\label{eq:V1_lower}\\
        V_2(x,t) &\leq \alpha_2(||x-x^*||)+D_K||\Delta(t)||,\label{eq:V2_upper}
    \end{align}
    for some $\alpha_1,\alpha_2\in\mathcal{K}_\infty$ and $D_K>0$ is the diameter of the compact and convex set $K$. Note that $x^*$ is the unique solution to $VI(K,F)$. Denote $\pi^*=F(x^*)$, then we have $       {x^*}^T \pi^* - \min_{y\in K} \ y^T \pi^*=0$. As a result, we have
    \begin{align}
        V_1(x)&=x^T \pi - \min_{y\in K} \ y^T \pi + {x^*}^T \pi^* - \min_{y\in K} \ y^T \pi^*\notag\\
        &=(x-x^*)^T(\pi-\pi^*) + ({x^*}^T\pi-\min_{y\in K} \ y^T\pi) \notag\\
        &\hspace{20pt} +(x^T\pi^*-\min_{y\in K}\ y^T\pi^*)\notag\\
        &\geq c||x-x^*||^2 := \alpha_1(||x-x^*||), \label{eq:V_lower}
    \end{align}
    where $\pi$ is attached to $F(x)$. On the other hand,
    \begin{align}
        V_2(&x,t)
        =x^T (\pi+\Delta(t)) - \min_{y\in K} \ y^T (\pi+\Delta(t)) \notag\\
        &\hspace{35pt}- {x^*}^T \pi^* + \min_{y\in K} \ y^T \pi^*\notag\\
        &=x^T(\pi+\Delta(t))-{x^*}^T\pi^*-y_1^T(\pi+\Delta(t))+y_2^T\pi^*\notag\\
        &\leq (x-x^*)^T(\pi-\pi^*)+(x-x^*)^T\pi^* \notag\\
        &\hspace{10pt}+{x^*}^T(\pi-\pi^*)-y_1^T(\pi-\pi^*)+(x-y_1)^T\Delta(t)\notag\\
        &\leq \underbrace{L||x-x^*||^2 + M||x-x^*||}_{:= \alpha_2(||x-x^*||)}+D_K||\Delta(t)||,\label{eq:V_upper}
    \end{align}
    where $\pi$ is attached to $F(x)$, $y_1=\arg\min_{y\in K}\ y^T(\pi+\Delta(t))$, and $y_2=\arg\min_{y\in K}\ y^T\pi^*$. The last inequality with $L,M>0$ follows from the Cauchy-Schwartz inequality, compactness of $K$, and Lipschitzness of $F$. Next, we prove
    \begin{align}
        ||V_2(x,t)-V_1(x)||\leq D_K ||\Delta(t)||. \label{eq:V1_V2_bound}
    \end{align}
    \begin{align*}
        V_2(x,t)-V_1(x)
        &=x^T\Delta(t) + y_1^T\pi - y_2^T(\pi+\Delta(t))\\
        &\leq x^T\Delta(t) + y_2^T\pi - y_2^T(\pi+\Delta(t))\\
        &\leq D_K||\Delta(t)||,
    \end{align*}
    where $\pi$ is attached to $F(x)$, $y_1=\arg\min_{y\in K} y^T\pi$, and $y_2=\arg \min_{y\in K} y^T(\pi+\Delta(t))$. Similarly, $V_2(x,t)-V_1(x) \geq (x-y_1)^T\Delta \geq -D_K||\Delta||$. Thus, (\ref{eq:V1_V2_bound}) is proved.
    
    In the following, we will derive a bound for the trajectory $V_2(x(t),t)$. We slightly abuse the notation by abbreviating $x(t)$, $\pi(t)$, $\Delta(t)$ as $x$, $\pi$, and $\Delta$, respectively.
    \begin{align*}
    \frac{d}{dt}&V_2(x(t),t)=\dot{x}^T(\pi+\Delta)+\left(x-\beta(\pi+\Delta)\right)^T(\dot{\pi}+\dot{\Delta})\\
    &=-V_2+\varepsilon^T(\pi+\Delta)+\left(x-\beta(\pi+\Delta)\right)^T\dot{\Delta}\\
    &\hspace{20pt}-\left(\beta(\pi+\Delta)-x\right)^T DF(x) \left(\beta(\pi+\Delta)-x+\varepsilon\right)\\
    &\leq -V_2-c||x-\beta(\pi+\Delta)||^2+||x-\beta(\pi+\Delta)|| ||\dot{\Delta}||\\
    &\hspace{20pt}+\sigma||x-\beta(\pi+\Delta)||||\varepsilon||+\varepsilon^T(\pi+\Delta)\\
    &\leq -V_2 + \underbrace{\frac{1}{2c}||\dot{\Delta}||^2 + \frac{\sigma^2}{2c}||\varepsilon||^2+M_1||\varepsilon||+||\varepsilon||||\Delta||}_{:=\Gamma(t)},  
    \end{align*}
    where $c$ is from \Cref{def:strong_mono}, $\sigma=||DF||_2$, and $M_1=\max_{x\in K}||\pi(x)||$. Then, by Comparison Lemma \cite{Khalil2002}, we get
    \begin{align*}
        V_2(x(t),t)
        &\leq V_2\left(x(0),0\right)e^{-t} + ||\Gamma||_\infty.
    \end{align*}
    Therefore by (\ref{eq:V1_V2_bound}), we have
    \begin{align}
        V_1(x(t))\leq V_2\left(x(0),0\right)e^{-t} +||\Gamma||_\infty + D_K||\Delta||_\infty. \label{eq:V1_bound}
    \end{align}    
    Finally, (\ref{eq:comb_bound}) follows from (\ref{eq:V1_lower}), (\ref{eq:V2_upper}), and (\ref{eq:V1_bound}).
\end{proof}

We next study a state-dependent perturbation $\delta$ on $F$:
\begin{align}
    \dot{x}(t)\in \beta(\tilde{\pi}(t))-x(t), \quad \tilde{\pi}(t)= (F+\delta)(x(t)),\label{eq:state_pay_dis}
\end{align}
Although we can analyze (\ref{eq:state_pay_dis}) with \Cref{thm:comb} by setting $\Delta(t)=\delta(x(t))$ and $\varepsilon(t)=0$, in the next theorem we exploit the additional structure to derive a refined 
bound.


\begin{thm}[State-dependent cost disturbance]\label{thm:state_pay_dis}
    Consider (\ref{eq:state_pay_dis}).  Let $K \subseteq\mathbb{R}^n$ be compact and convex, $F:K \rightarrow \mathbb{R}^n$ be $C^1$ strongly monotone, and $\delta:K \rightarrow \mathbb{R}^n$ be such that $\tilde{\pi}$ is $C^1$ strongly monotone. Then, the dynamics converge to a new perturbed equilibrium point $\tilde{x}^*$ and
    \begin{align}
        ||x^*-\tilde{x}^*||\leq \alpha_1^{-1}\left(h(\tilde{x}^*)\right),\label{eq:state_pay_bound}
    \end{align}
    where $x^*$ is the unique solution of unperturbed $VI(K,F)$ and $\alpha_1\in\mathcal{K}_\infty$. The function $h$ is given as
    \begin{align*}
    h(x)=\max\left\{(z-x)^T\delta(x): z=\arg\min_{y\in K}\ y^T F(x)\right\}.
    \end{align*}
\end{thm}

\begin{proof}
    The proof is similar to Theorem \ref{thm:comb}. Define
    \begin{align}
        V_1(x)&=U(x,F(x))\\
        V_2(x)&=U(x,(F+\delta)(x))
    \end{align}
    where
    \begin{align}
        U(x,\pi)&=x^T \pi - \min_{y\in K}\ y^T\pi.
    \end{align}
    Refer to the prove of (\ref{eq:V1_lower}) and (\ref{eq:V1_V2_bound}), then we have
    \begin{align*}
        \alpha_1(||x-x^*||)\leq V_1(x)&\leq V_2(x) + (z-x)^T\delta(x),
    \end{align*}
    where $z=\arg\min_{y\in K}\ y^T\pi$. By Theorem \ref{thm:convex}, with the dynamics (\ref{eq:state_pay_dis}), $x(t)$ converges to a perturbed equilibrium $\tilde{x}^*$ satisfying $V_2(\tilde{x}^*)=0$. As a result, (\ref{eq:state_pay_bound}) follows. 
\end{proof}

\begin{rem}
    The model (\ref{eq:state_pay_dis})
    encompasses the perturbed best response dynamics, which was used in \cite{Hofbauer2007}
   to study a smooth approximation   of the best response dynamics. 
    The perturbed best response dynamics is defined as
    \begin{align*}
        \dot{x}(t) \in \tilde{B}(x(t)) - x(t), \tilde{B}(x(t))=\arg\min_{y\in X} y^T\pi(t) + H(y),
    \end{align*}
    where $\pi(t)=F(x(t))$ and $H$ is 
    strictly convex, twice differentiable, and with its magnitude of gradient approaching infinity near the boundary of $X$. For the optimization problem $\tilde{B}$, if instead of directly picking the $\arg\min$, we apply the Frank-Wolfe method, then the point selected is
    \begin{align*}
        \arg\min_{y\in X}\ y^T\left(F(x(t))+\frac{\partial H}{\partial y}(x(t))\right),
    \end{align*}
    which is 
    covered by (\ref{eq:state_pay_dis}) with $\delta(x(t))=\frac{\partial H}{\partial y}(x(t))$. Since $H$ is strictly convex, $\delta$ is strictly monotone. Therefore, \Cref{thm:state_pay_dis} reproduces the convergence results for perturbed best response dynamics \cite{Hofbauer2007}. In addition, it provides a bound on the perturbed distance.
\end{rem}

\section{APPLICATIONS}\label{sec:APP}

\subsection{Traffic 
network
}

Consider the network 
in Fig. \ref{fig:traffic_lights},
which includes 3 routes and 3 Y-intersections with traffic lights shown in boxes with crosses. The flows on each route are denoted $x_1$, $x_2$, and $x_3$, respectively, and their sum is one. We zoom in on one of the Y-intersections, shown in Fig. \ref{fig:Y}, to explain the effects of traffic lights. 
An actuated 
traffic light results in the interdependency of Link 1 and Link 2 delays on each others' flows. Assuming the light prioritizes the branch line,
we let $\Phi_1(x_1,x_2)=x_1+3x_2$ and $\Phi_2(x_2,x_1)=x_1+x_2$. For the rest of the delay functions in the network, we model similarly as $\Phi_3(x,y)=\Phi_5(x,y)=x+3y$ and $\Phi_4(x,y)=\Phi_6(x,y)=x+y$. Then, the total delays on each route are $2x_1+3x_2+x_3$, $x_1+2x_2+3x_3$, and $3x_1+x_2+2x_3$, respectively.  Thus, the equilibrium of the traffic network is characterized by the variational inequality  (\ref{eq:VI}) where $K$ is the probability simplex and
\begin{align*}
    F(x)=\begin{bmatrix}
        2x_1+3x_2+x_3\\
        x_1+2x_2+3x_3\\
        3x_1+x_2+2x_3
    \end{bmatrix}.
\end{align*}
Note that this is not a potential game but a monotone game, since $DF$ is not symmetric but $DF+DF^T$ is positive semidefinite.
The simulation result is provided in Fig. \ref{fig:traffic_wo} where the best response dynamics lead to a spiral trajectory and converge to the equilibrium point $[\frac{1}{3},\frac{1}{3},\frac{1}{3}]^T$. 

\begin{figure}[!ht]
    \centering
    \includegraphics[scale=.4]{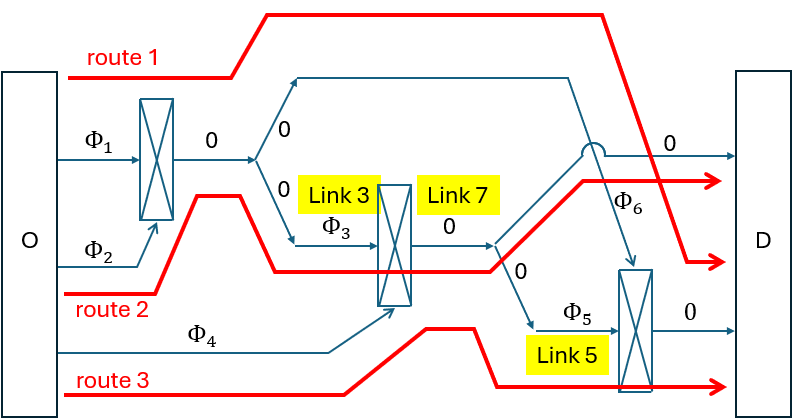}
    \caption{The network is composed of three Y-intersections. Three possible routes are given. The delay functions are provided in the contexts.}
    \label{fig:traffic_lights}
\end{figure}

\begin{figure}[!ht]
    \centering
    \includegraphics[scale=.35]{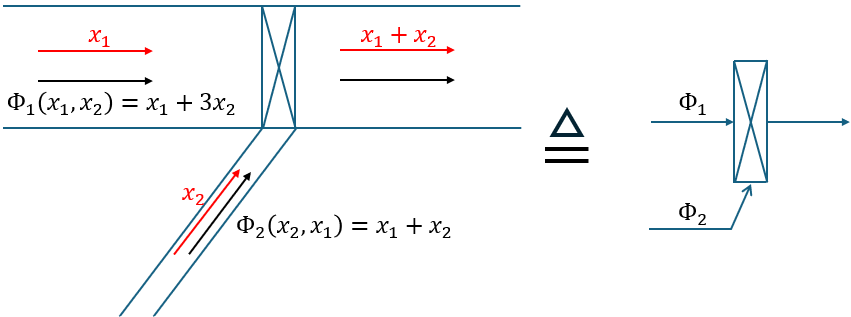}
    \caption{The figure describes the Y-intersection where the traffic lights induce delay. The traffic flows are in red, while the delay functions are in black. The branch line, Link 2, and the mainline, Link 1, affect each other through the traffic light making their delay functions depend on $x_1$ and $x_2$.}
    \label{fig:Y}
\end{figure}

As an illustration of the case where $K$ is a subset of the simplex, we 
add a
capacity constraint on Link 7: $x_2+x_3\leq0.9$, and delay constraints on Link 3 and Link 5: $\Phi_3(x_2,x_3)=x_2+3x_3\leq2$ and $\Phi_5(x_3,x_1)=x_3+3x_1\leq2$. 
Although the feasible set is no longer a simplex,
\Cref{thm:convex} ensures convergence. Simulation results are given in Fig. \ref{fig:traffic_w}.

\begin{figure}[!ht]
    \centering
    \begin{subfigure}[t]{0.23\textwidth}
    \includegraphics[width=\textwidth]{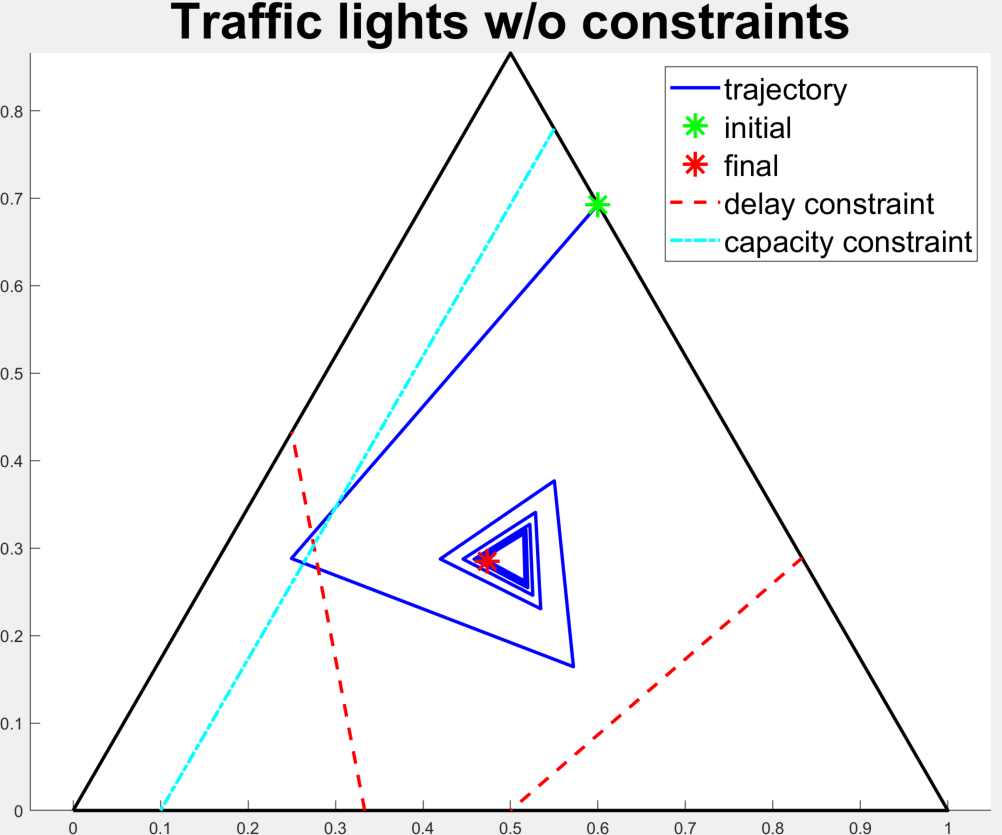}
    \caption{without constraints}
    \label{fig:traffic_wo}
    \end{subfigure}
    \hfill
    \begin{subfigure}[t]{0.23\textwidth}
    \includegraphics[width=\textwidth]{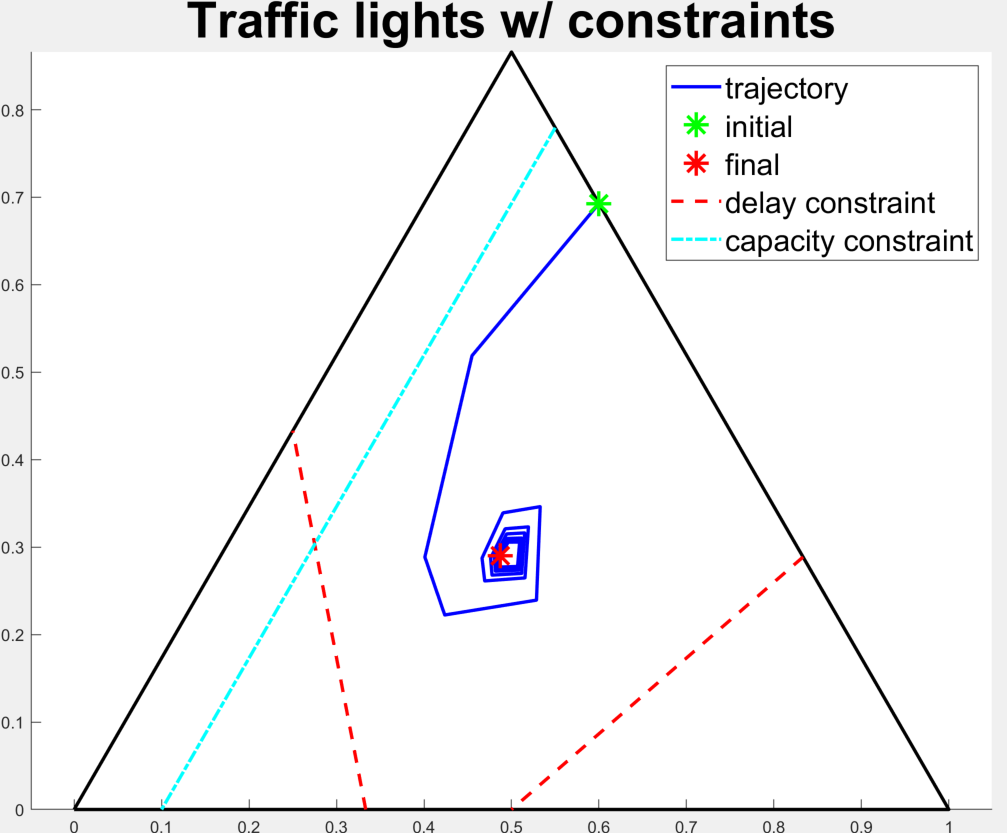}
    \caption{with constraints}
    \label{fig:traffic_w}
    \end{subfigure}
    \caption{Since the trajectory evolves on the 3-dimensional probability simplex, we project it and draw it on a plane. Compared with Fig. \ref{fig:traffic_wo},  the trajectory in Fig. \ref{fig:traffic_w} changes drastically and remains in the feasible set.}
\end{figure}

\subsection{Cost disturbances and dynamics disturbances} \label{sec:APP_dis}

In the following, we first discuss congestion games with cost disturbances and then dynamics disturbances.

\subsubsection{Cost disturbances}\label{sec:APP_dis_pay}

In standard 
congestion games, the delays are merely based on the flow. However, there 
may be
other time-varying factors that affect the delays. For example, weather conditions cause additional delays,
and accidents increase the delays 
for a period until the road is cleared.
To capture these scenarios, we introduce a bounded time-varying cost disturbance $\Delta(t)$ with bounded derivative $\dot{\Delta}(t)$. 

A congestion game is strongly monotone if all the delay functions are strictly increasing. In this case, by Theorem \ref{thm:comb}, the convergence results of best response dynamics are robust to a certain extent of cost disturbances. The settings of the simulation are described in Fig. \ref{fig:congestion_net}, where the corresponding $F$ for the variational inequality is $
    F(x) = [2x_1+0.5x_3+0.3, 1.5x_2+0.5x_3+0.5, 0.5x_1+0.5x_2+2x_3+0.6]^T$.
\begin{figure}[!ht]
    \begin{minipage}[t]{.23\textwidth}
    \centering
    \includegraphics[scale=.4]{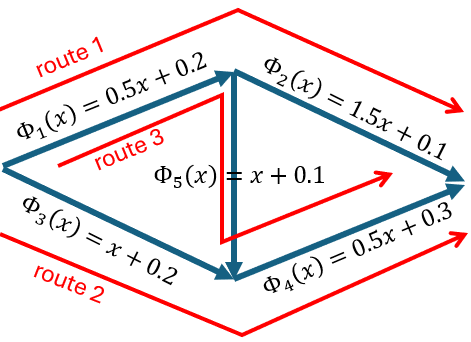}
    \caption{delay functions for links}
    \label{fig:congestion_net}
    \end{minipage}
    \hfill
    \begin{minipage}[t]{.23\textwidth}
    \centering
    \includegraphics[scale=.195]{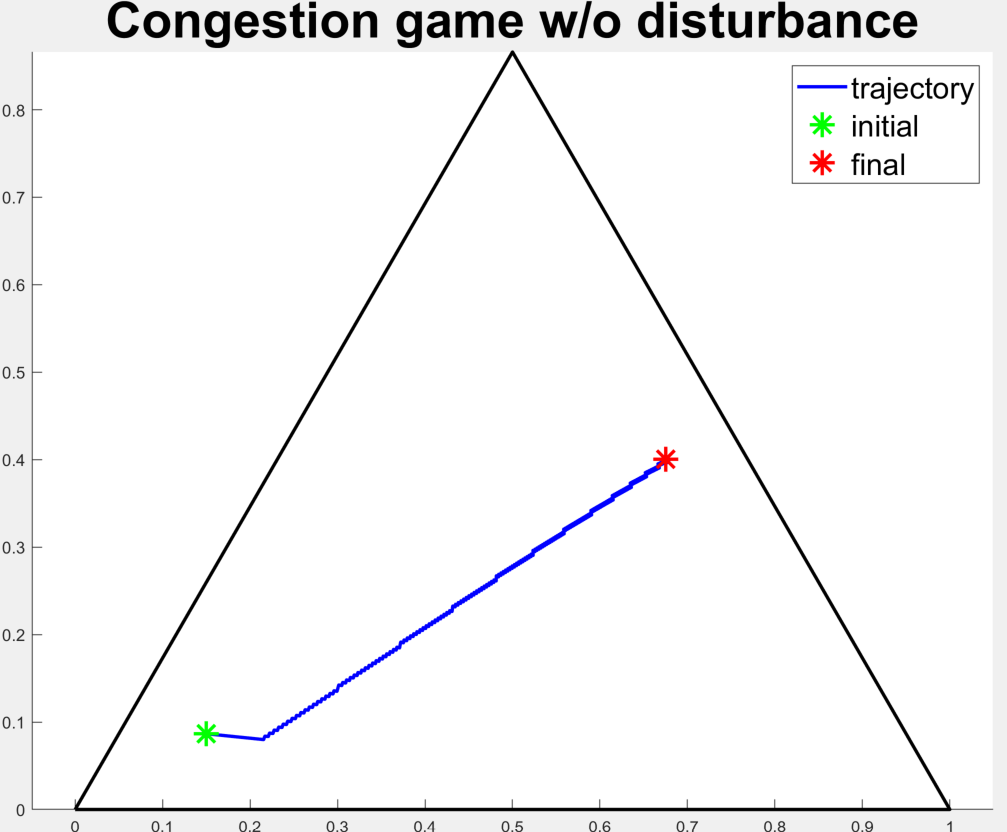}
    \caption{without disturbance}
    \label{fig:congestion_wo_pay_dis}        
    \end{minipage}
\end{figure}

The results of solving the variational inequality by the best response dynamics without cost disturbances are provided in Fig. \ref{fig:congestion_wo_pay_dis}. Note that the function $F$ is strongly monotone and thus admits a unique Nash equilibrium.

Next, we give simulation results with two examples of cost disturbances.
The first one is the periodic cost disturbance, $\Delta_1(t)=[0.1\sin(0.01t), -0.05\cos(0.02t+10), 0.15\sin(0.05t-20)]^T$. The result is shown in Fig. \ref{fig:congestion_w_pay_dis}. Since 
this
disturbance is periodic and does not vanish, the trajectory 
reaches a neighborhood of
the unperturbed equilibrium rather than the equilibrium itself. The second example is the diminishing cost disturbance, $\Delta_2(t)=20e^{-0.01t}\Delta_1(t)$. The simulation result is provided in Fig. \ref{fig:congestion_w_pt0}. Since 
the disturbance is large in the early stage, the trajectory 
moves
away from the original trajectory initially. However, as the disturbance vanishes, the trajectory gradually converges to the non-disturbed one as in Fig. \ref{fig:congestion_wo_pay_dis}.

\begin{figure}[!ht]
    \centering
    \includegraphics[scale=.22]{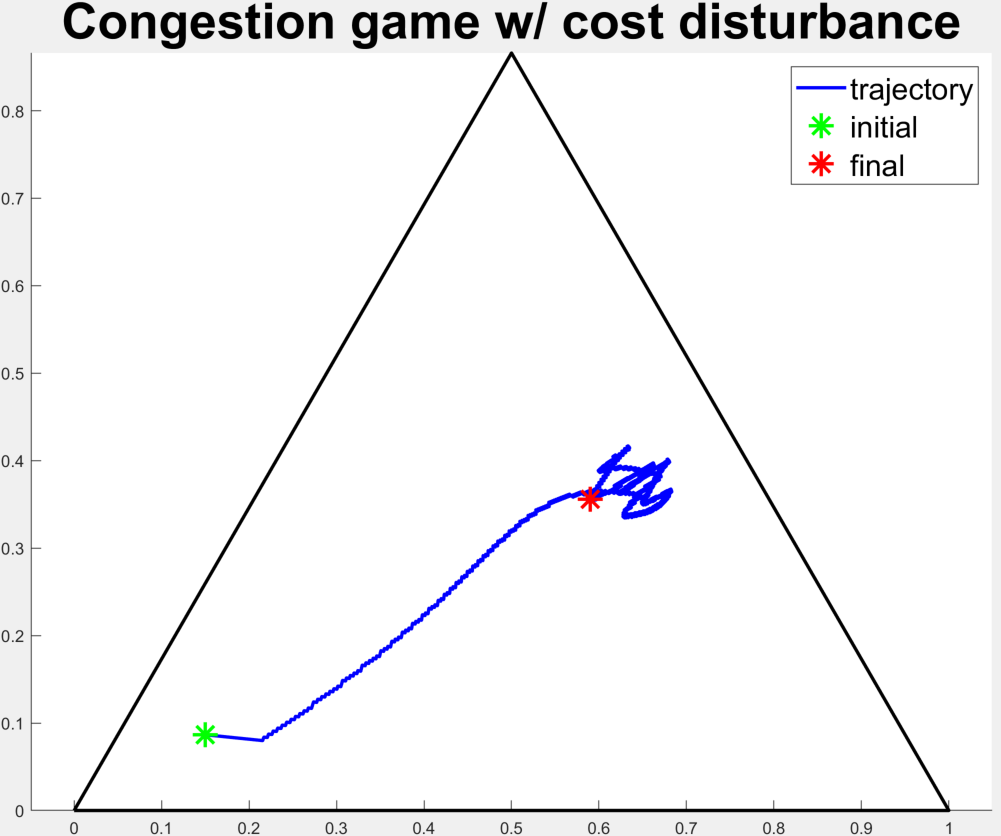}
    \caption{A periodic cost disturbance is injected. The state trajectory moves around the unperturbed equilibrium as in Fig. \ref{fig:congestion_wo_pay_dis}.}
    \label{fig:congestion_w_pay_dis}
\end{figure}
\begin{figure}[!ht]
    \centering
    \includegraphics[scale=.22]{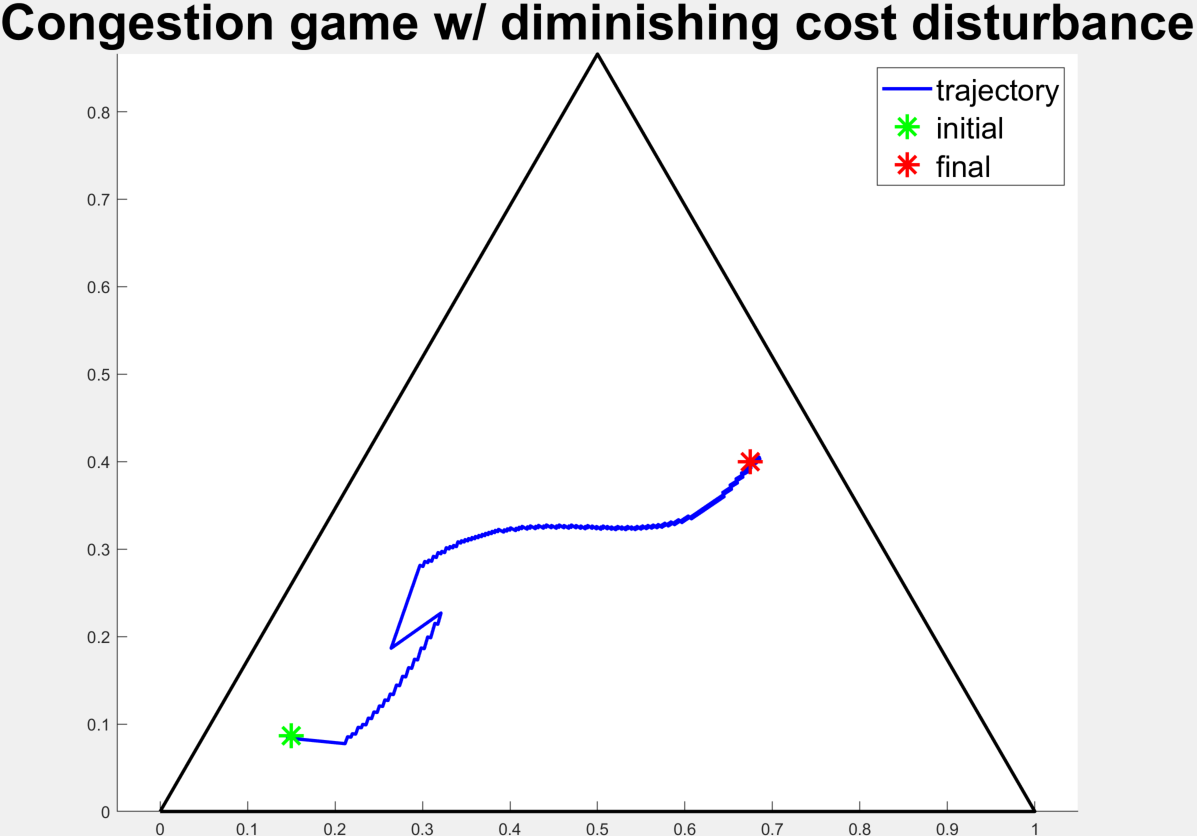}
    \caption{A diminishing cost disturbance is introduced. Initially, when the disturbance is large, the trajectory is driven away from the original unperturbed trajectory. However, as the disturbance vanishes, the trajectory gradually converges to the unperturbed one as in Fig. \ref{fig:congestion_wo_pay_dis}}
    \label{fig:congestion_w_pt0}
\end{figure}

\subsubsection{Dynamics disturbances}\label{sec:APP_dis_dyn}

Next we simulate the system with a dynamics disturbance $\varepsilon(t)=[0.7\sin(0.1t), 0.7\cos(0.2t-10), -0.7\sin(0.1t)-0.7\cos(0.2t-10)]^T$. Note that $\varepsilon(t)$ is 
admissible as defined in Definition \ref{def:admissible}. By Theorem \ref{thm:comb}, the resulting trajectory should remain close to the unperturbed one as in Fig. \ref{fig:congestion_wo_pay_dis}. The simulation results are provided in Fig. \ref{fig:congestion_w_dyn_dis}.
\begin{figure}[!ht]
    \centering
    \includegraphics[scale=.23]{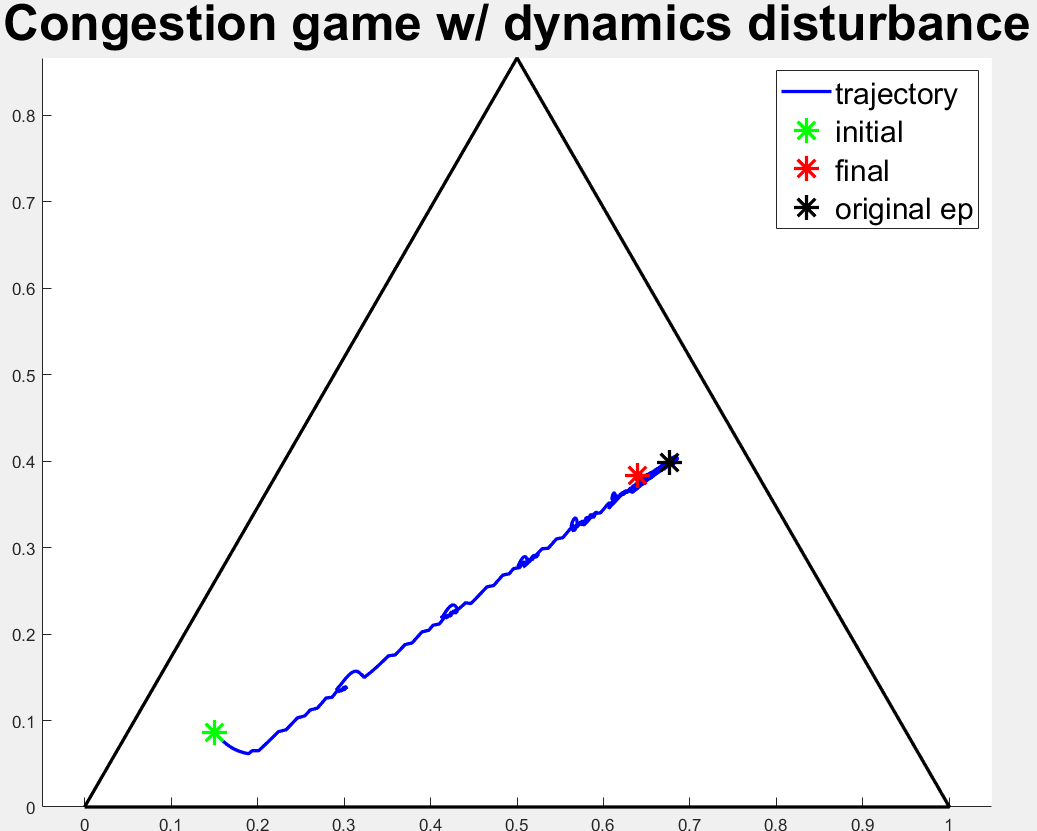}
    \caption{A periodic dynamics disturbance is added. The resulting trajectory remains close to the unperturbed one as in Fig. \ref{fig:congestion_wo_pay_dis}.}
    \label{fig:congestion_w_dyn_dis}
\end{figure}

\section{CONCLUSIONS AND FUTURE WORK}\label{sec:Conclusion}

In this paper, we leveraged the best response dynamics to solve monotone variational inequalities and provided robustness guarantees against classes of disturbances.
When the set $K$ is a subset of the probability simplex, as in our example, $x(t)$ can be viewed as the distribution of the agents to the strategies in a population game with constraints. In this special case, it would be possible to study learning rules other than the best response, such as those reviewed in \cite{Park2019CDC}.
Note, however, that our convergence results apply to any compact and convex set $K$; therefore, they are not restricted to subsets of the standard simplex used in population games.
This flexibility allows for addressing population games where the entries of $x(t)$ do not necessarily sum to a constant. This would allow ``outside options," {\it e.g.}, choosing not to drive in a congestion game. Future research will address such scenarios as well as the incorporation of other learning rules.




\bibliographystyle{IEEEtran}
\bibliography{reference.bib}

\end{document}